\documentclass[11pt,a4paper,reqno]{amsart}
\usepackage[english]{babel}
\usepackage[applemac]{inputenc}
\usepackage[T1]{fontenc}
\usepackage{palatino}
\usepackage{amsmath}
\usepackage{amssymb}
\usepackage{amsthm}
\usepackage{amsfonts}
\usepackage{graphicx}
\usepackage{mathtools}

\DeclarePairedDelimiter\floor{\lfloor}{\rfloor}
\usepackage[colorlinks = true, citecolor = black]{hyperref}
\author{Tuomas Orponen}
\title{On the distance sets of AD-regular sets}
\address{University of Helsinki, Department of Mathematics and Statistics}
\subjclass[2010]{28A80 (Primary) 28A78 (Secondary)}
\thanks{T.O. is supported by the Academy of Finland through the grant Restricted families of projections and connections to Kakeya type problems.}
\email{tuomas.orponen@helsinki.fi}

\newcommand{\R}{\mathbb{R}}
\newcommand{\N}{\mathbb{N}}

\newcommand{\calD}{\mathcal{D}}
\newcommand{\calH}{\mathcal{H}}

\newcommand{\calF}{\mathcal{F}}

\newcommand{\spt}{\operatorname{spt}}
\newcommand{\Hd}{\dim_{\mathrm{H}}}
\newcommand{\Pd}{\dim_{\mathrm{p}}}
\newcommand{\Bd}{\overline{\dim}_{\mathrm{B}}}
\newcommand{\calP}{\mathcal{P}}

\newcommand{\calE}{\mathcal{E}}

\newcommand{\diam}{\operatorname{diam}}

\newcommand{\dist}{\operatorname{dist}}

\newcommand{\interior}{\operatorname{int}}

\numberwithin{equation}{section}

\theoremstyle{plain}
\newtheorem{thm}[equation]{Theorem}

\newtheorem{lemma}[equation]{Lemma}

\newtheorem{cor}[equation]{Corollary}
\newtheorem{proposition}[equation]{Proposition}

\theoremstyle{definition}

\newtheorem{definition}[equation]{Definition}

\theoremstyle{remark}
\newtheorem{remark}[equation]{Remark}

\begin{document}

\begin{abstract} I prove that if $\emptyset \neq K \subset \R^{2}$ is a compact $s$-Ahlfors-David regular set with $s \geq 1$, then 
\begin{displaymath} \Pd D(K) = 1, \end{displaymath}
where $D(K) := \{|x - y| : x,y \in K\}$ is the distance set of $K$, and $\Pd$ stands for packing dimension.

The same proof strategy applies to other problems of similar nature. For instance, one can show that if $\emptyset \neq K \subset \R^{2}$ is a compact $s$-Ahlfors-David regular set with $s \geq 1$, then there exists a point $x_{0} \in K$ such that $\Pd K \cdot (K - x_{0}) = 1$. Specialising to product sets, one derives the following sum-product corollary: if $A \subset \R$ is a non-empty compact $s$-Ahlfors-David regular set with $s \geq 1/2$, then 
\begin{displaymath} \Pd [A(A - a_{1}) + A(A - a_{2})] = 1 \end{displaymath}
for some $a_{1},a_{2} \in A$. In particular, $\Pd [AA + AA - AA - AA] = 1$. In all of the results mentioned above, compactness can be relaxed to boundedness and $\calH^{s}$-measurability, if packing dimension is replaced by upper box dimension.
\end{abstract}

\maketitle

\section{Introduction}

Given a planar set $K$, the \emph{distance set problem} asks for a relationship between the size of $K$, and the size of the distance set
\begin{displaymath} D(K) := \{|x - y| : x,y \in K\}. \end{displaymath}
For finite sets $K$, the problem is due to P. Erd\H{o}s from 1946, and the \emph{Erd\H{o}s distance conjecture} states that the cardinality of $D(K)$ should satisfy $|D(K)| \gtrsim |P|/\sqrt{\log |P|}$. L. Guth and N. Katz \cite{GK} nearly resolved the question in 2011 by showing that $|D(K)| \gtrsim |P|/\log |P|$. 

The "continuous" version of the distance set problem was proposed by K. Falconer \cite{Fa} in 1985. The \emph{Falconer distance conjecture} claims that if $K \subset \R^{2}$ is a Borel set with $\dim K > 1$, then $D(K)$ has positive length. As far as I know, the current records in this setting are the following theorems of T. Wolff \cite{Wo} from 1999 and J. Bourgain \cite{Bo} from 2003:
\begin{thm}[Wolff] If $K \subset \R^{2}$ is Borel with $\dim K > 4/3$, then $D(K)$ has positive length. \end{thm}
\begin{thm}[Bourgain] If $K \subset \R^{2}$ is Borel with $\dim K \geq 1$, then $\Hd D(K) \geq 1/2 + \epsilon$ for some (small) absolute constant $\epsilon > 0$. \end{thm} 
In Bourgain's result, $\Hd$ stands for Hausdorff dimension. The following theorem is the main result of this note:

\begin{thm}\label{main} Assume that $\emptyset \neq K \subset \R^{2}$ is a compact $s$-Ahlfors-David-regular set with $s \geq 1$. Then 
\begin{displaymath} \Pd D(K) = 1, \end{displaymath}
where $D(K) = \{|x - y| : x,y \in K\}$ is the distance set of $K$.
\end{thm}

The proof of Theorem \ref{main} can be easily modified to give various other results of similar nature. I have Alex Iosevich to thank for the following suggestions:
\begin{thm}\label{dotProducts} Assume that $\emptyset \neq K \subset \R^{2}$ is a compact $s$-Ahlfors-David-regular set with $s \geq 1$. Then, there exists a point $x_{0} \in K$ such that
\begin{displaymath} \Pd K \cdot (K - x_{0}) = \Pd \{x_{1} \cdot (x_{2} - x_{0}) : x_{1},x_{2} \in K\} = 1. \end{displaymath}
\end{thm}

\begin{cor}\label{sumProduct} Assume that $\emptyset \neq A \subset \R$ is a compact Ahlfors-David regular set with $\Hd A \geq 1/2$. Then, there exist points $a_{1},a_{2} \in A$ such that
\begin{displaymath} \Pd [A(A - a_{1}) + A(A - a_{2})] = 1. \end{displaymath}
\end{cor}

Theorem \ref{main} is proved in Section \ref{mainProof}, after the necessary tools have been developed in Section \ref{preliminaries}. The small modifications needed to prove Theorem \ref{dotProducts} and Corollary \ref{sumProduct} are discussed in Section \ref{corollaries}. We conclude the introduction by defining the some basic concepts.

\begin{definition}[Packing and box dimensions] Above, $\Pd$ stands for packing dimension. To define it, we first write
\begin{displaymath} \Bd A := \limsup_{\delta \to 0} \frac{\log N(A,\delta)}{-\log \delta} \end{displaymath}
for bounded sets $A \subset \R^{d}$, where $N(A,\delta)$ is the least number of $\delta$-balls required to cover $A$. The dimension $\Bd A$ is commonly known as upper box dimension or upper Minkowski dimension.  Then, $\Pd$ is defined by
\begin{displaymath} \dim_{\textup{p}}(B) = \inf \left\{ \sup \Bd F_{i} : B \subset \bigcup_{i} F_{i}, \: F_{i} \text{ closed}\right\}. \end{displaymath} 
Since $\Bd A = \Bd \bar{A}$ for all bounded sets $A$, one has $\Pd A \leq \Bd A$ for all bounded sets $A$. The converse need not hold even for compact sets, since
\begin{displaymath} \Pd \{0,1,\tfrac{1}{2},\tfrac{1}{3},\ldots\} = 0 < 1/2 = \Bd \{0,1,\tfrac{1}{2},\tfrac{1}{3},\ldots\}. \end{displaymath}
\end{definition}

\begin{definition}[AD-regular sets and measures] A Borel measure $\mu$ on $\R^{d}$ is said to be $(s,A)$-Ahlfors-David regular -- or $(s,A)$-AD-regular in short -- if
\begin{displaymath} \frac{r^{s}}{A} \leq \mu(B(x,r)) \leq Ar^{s} \end{displaymath}
for all $x \in \spt \mu$ and $0 < r \leq \diam(\spt \mu)$. An $\calH^{s}$-measurable set $K \subset [0,1]$ is called $(s,A)$-AD regular, if $0 < \calH^{s}(K) < \infty$, and the restriction $\mu := \calH^{s}|_{K}$ of $\calH^{s}$ to $K$ is $(s,A)$-AD regular.
\end{definition}

\section{Acknowledgements} I am grateful to Alex Iosevich for pointing out that Theorem \ref{dotProducts} and Corollary \ref{sumProduct} follow from the proof strategy employed in Theorem \ref{main}. In fact, an earlier version of the manuscript claimed even stronger versions of Theorem \ref{dotProducts} and Corollary \ref{sumProduct}, but Pablo Shmerkin pointed out an error in the argument; I'm very thankful for his careful reading! 

\section{Preliminaries on entropy and projections}\label{preliminaries} Many of the arguments in this section are repeated from \cite{O}, where, further, the discussion closely followed that of M. Hochman's paper \cite{Ho}. 

\begin{definition}[Measures and their blow-ups in $\R^{d}$] Let $\calP(\Omega)$ stand for the space of Borel probability measures on $\Omega$. In what follows, $\Omega$ will be $\R^{d}$, or a cube in $\R^{d}$, and $d \in \{1,2\}$. If $Q = r[0,1)^{d} + a$ is a cube in $\R^{d}$, let $T_{Q}(x) := (x - a)/r$ be the unique homothety taking $Q$ to $[0,1)^{d}$. Given a measure $\mu \in \calP(\R^{d})$ and a cube $Q$ as above, with $\mu(Q) > 0$, define the measures
\begin{displaymath} \mu_{Q} := \frac{1}{\mu(Q)} \mu|_{Q} \quad \text{and} \quad \mu^{Q} := T_{Q\sharp}(\mu_{Q}), \end{displaymath} 
where $\mu|_{Q}$ is the restriction of $\mu$ to $Q$, and $T_{Q\sharp}$ is the push-forward under $T_{Q}$. So, $\mu^{Q}$ is a "blow-up" of $\mu_{Q}$ into $[0,1)^{d}$.
\end{definition}

\begin{definition}[Entropy]\label{entropy} Let $\mu \in \calP(\Omega)$, and let $\calF$ be a countable $\mu$-measurable partition of $\Omega$. Set
\begin{displaymath} H(\mu,\calF) := -\sum_{F \in \calF} \mu(F) \log \mu(F), \end{displaymath}
where the convention $0 \cdot \log 0 := 0$ is used. If $\calE$ and $\calF$ are two $\mu$-measurable partitions, one also defines the conditional entropy
\begin{displaymath} H(\mu,\calE | \calF) := \sum_{F \in \calF} \mu(F)  \cdot H(\mu_{F},\calE), \end{displaymath}
where $\mu_{F} := \mu|_{F}/\mu(F)$, if $\mu(F) > 0$. 
 \end{definition} 
 
 The notion of conditional entropy is particularly useful, when $\calE$ \emph{refines} $\calF$, which means that every set in $\calE$ is contained in a (unique) set in $\calF$:
 \begin{proposition}[Conditional entropy formula]\label{CEF} Assume that $\calE,\calF$ are partitions as in Definition \ref{entropy}, and $\calE$ refines $\calF$. Then
 \begin{displaymath} H(\mu,\calE|\calF) = H(\mu,\calE) - H(\mu,\calF). \end{displaymath}
 In particular, $H(\mu,\calE) \geq H(\mu,\calF)$.
 \end{proposition} 
 
 \begin{proof} For $F \in \calF$, let $\calE(F) := \{E \in \calE : E \subset F\}$. A direct computation gives
 \begin{align*} H(\mu,\calE | \calF) & = - \sum_{F \in \calF} \mu(F) \cdot \sum_{E \in \calE} \mu_{F}(E) \log \mu_{F}(E)\\
 & = - \sum_{F \in \calF} \sum_{E \in \calE(F)} \mu(E) \log \frac{\mu(E)}{\mu(F)}\\
 & = - \left( \sum_{E \in \calE} \mu(E) \log \mu (E) - \sum_{F \in \calF} \log \mu(F) \sum_{E \in \calE(F)} \mu(E) \right)\\
 & = H(\mu,\calE) + \sum_{F \in \calF} \mu(F) \log \mu(F) = H(\mu,\calE) - H(\mu,\calF), \end{align*}
as claimed.  \end{proof} 

The partitions $\calE,\calF$ used below will be the dyadic partitions of $\R^{d}$: $\calE,\calF = \calD_{n}$. The lemma below contains three more useful and well-known -- or easily verified -- properties of entropy. The items are selected from \cite[Lemma 3.1]{Ho} and \cite[Lemma 3.2]{Ho}.
\begin{lemma}\label{factsOfLife} Let $\calE, \calF$ be countable $\mu$-measurable partitions of $\Omega$.
\begin{itemize} 
\item[(i)] The functions $\mu \mapsto H(\mu,\calE)$ and $\mu \mapsto H(\mu,\calE | \calF)$ are concave.
\item[(ii)] If $\spt \mu \subset B(0,R)$, and $f,g \colon B(0,R) \to \R$ are functions so that $|f(x) - g(x)| \leq R2^{-n}$ for $x \in B(0,R)$, then 
\begin{displaymath} |H(f_{\sharp}\mu,\calD_{n}) - H(g_{\sharp}\mu,\calD_{n})| \leq C, \end{displaymath}
where $C > 0$ only depends on $R$.
\item[(iii)] If every set in $\calE$ meets at most $C \geq 1$ sets in $\calF$ and vice versa, then 
\begin{displaymath} |H(\mu,\calE) - H(\mu,\calF)| \lesssim_{C} 1. \end{displaymath}
\end{itemize}
\end{lemma} 

Let $\calD_{n}$ be the family of dyadic cubes of side-length $2^{-n}$ in $\R^{d}$ (the notation will be used in both $\R^{2}$ and $\R$). For $n \in \N$, write $H_{n}$ for the \emph{normalised scale $2^{-n}$-entropy}
\begin{displaymath} H_{n}(\mu) := \frac{1}{\log 2^{n}} \cdot H(\mu,\calD_{n}) = \sum_{Q \in \calD_{n}} \mu(Q) \cdot \left(\frac{\log \mu(Q)}{\log 2^{-n}} \right). \end{displaymath}
Now, all the definitions and tools are in place to state and prove the key auxiliary result from Hochman's paper, namely \cite[Lemma 3.5]{Ho}, in slightly modified form:
\begin{lemma}\label{multiScale} Let $\mu \in \calP([0,1]^{2})$ and $m,n \in \N$ with $m < n$. Then, for any continuous mapping $f \colon [0,1]^{2} \to \R$, 
\begin{displaymath} H_{n}(f_{\sharp}\mu) \geq \frac{1}{n} \sum_{k = 0}^{\floor*{n/m} - 1} \sum_{Q \in \calD_{km}} \mu(Q) \cdot H(f_{\sharp}\mu_{Q}, \calD_{(k + 1)m} | \calD_{km}). \end{displaymath}
The sum $Q \in \calD_{km}$ only runs over those $Q$ with $\mu(Q) > 0$.
\end{lemma}

\begin{proof} Write $n = k_{0}m + r$, where $0 \leq r < m$, and $k_{0} = \floor*{n/m}$. Then
\begin{align*} H(f_{\sharp}\mu,\calD_{n}) \geq H(f_{\sharp}\mu,\calD_{k_{0}m}) & = \sum_{k = 0}^{k_{0} - 1} H(f_{\sharp}\mu, \calD_{(k + 1)m} | \calD_{km}) + H(f_{\sharp}\mu,\calD_{0})\\
& \geq \sum_{k = 0}^{k_{0} - 1} H(f_{\sharp}\mu, \calD_{(k + 1)m} | \calD_{km}) \end{align*} 
by repeated application of Proposition \ref{CEF}. Next, since $f_{\sharp} \colon \calP([0,1]^{2}) \to \calP(\R)$ is linear (even if $f$ is not), one has
\begin{align*} f_{\sharp}\mu & = f_{\sharp} \left(\sum_{Q \in \calD_{km}} \mu|_{Q} \right) = \sum_{Q \in \calD_{km}} f_{\sharp}(\mu|_{Q}) = \sum_{Q \in \calD_{km}} \mu(Q) \cdot f_{\sharp} \mu_{Q}, \end{align*}
so, by Jensen's inequality and the concavity of (conditional) entropy,
\begin{displaymath} H(f_{\sharp}\mu,\calD_{(k + 1)m} | \calD_{km}) \geq \sum_{Q \in \calD_{km}} \mu(Q) \cdot H(f_{\sharp}\mu_{Q}, \calD_{(k + 1)m} | \calD_{km}). \end{displaymath}
Dividing by $n$ completes the proof. \end{proof} 

\begin{remark} In case $f$ is linear, say $f = \pi_{e}$ for some $e \in S^{1}$, where $\pi_{e}$ is the orthogonal projection $\pi_{e}(x) = x \cdot e$, the lemma can be taken a step further. Observe that
\begin{displaymath} H(\pi_{e\sharp}\mu_{Q},\calD_{(k + 1)m}|\calD_{km}) = H(\pi_{e\sharp}\mu^{Q}, \calD_{m} | \calD_{0}) = H(\pi_{e\sharp}\mu^{Q},\calD_{m}) - H(\pi_{e\sharp}\mu^{Q}, \calD_{0}), \end{displaymath}
by the linearity of $\pi_{e}$ and Proposition \ref{CEF}. Here $H(\pi_{e\sharp}\mu^{Q},\calD_{0}) \leq 3$, because $\pi_{e\sharp}\mu^{Q}$ is supported in an interval of length $\sqrt{2}$. So,
\begin{equation}\label{form17} H(\pi_{e\sharp}\mu_{Q},\calD_{(k + 1)m}|\calD_{km}) \geq m \cdot H_{m}(\pi_{e\sharp}\mu^{Q}) - 3. \end{equation}
\end{remark}

\subsection{An entropy version of Marstrand's theorem}

To estimate the quantity on the right hand side of \eqref{form17} from below, one needs the following Marstrand type projection result:

\begin{proposition}\label{entropyMarstrand} Assume that $\mu \in \calP([0,1]^{2})$ satisfies the linear growth condition $\mu(B(x,r)) \leq Ar$ for $x \in \R^{2}$, $r > 0$ and some $A \geq 1$. Then
\begin{displaymath} \int_{S^{1}} H_{m}(\pi_{e\sharp}\mu) \, d\sigma(e) \geq s - ACm \cdot 2^{(s - 1)m}, \qquad 0 < s < 1, \end{displaymath} 
where $\sigma$ is the unit-normalised length measure on $S^{1}$, and $C > 0$ is an absolute constant.
\end{proposition}

\begin{proof} Fix $m \in \N$. It follows from the linear growth condition for $\mu$ that
\begin{equation}\label{form12} \int_{S^{1}} 2^{m} \sum_{Q \in \calD_{m}} [\pi_{e\sharp}\mu(Q)]^{2} \, d\sigma(e) \lesssim Am. \end{equation} 
This is standard, so I only sketch the details: observe that for any $\nu \in \calP([0,1]^{2})$
\begin{align*} \int_{S^{1}} \|\pi_{e\sharp}\nu\|_{2}^{2} \, d\sigma(e) & = \int_{S^{1}} \int_{\R} |\hat{\nu}(te)|^{2} \, dt \, d\sigma(e)\\
& \sim \int_{\R^{2}} |\hat{\nu}(\xi)|^{2}|\xi|^{-1} \, d\xi \sim \iint \frac{d\nu x \, d\nu y}{|x - y|} =: I_{1}(\nu). \end{align*} 
Apply this with $\nu := \mu \ast \psi_{m}$, where $\psi_{m}(x) := 2^{2m}\psi(2^{m}x)$ and $\psi$ is a radial bump function with $\chi_{B(0,5)} \leq \psi \leq \chi_{B(0,10)}$. Using the linear growth condition for $\mu$, it is easy to verify that $I_{1}(\mu \ast \psi_{m}) \lesssim Am$, for $A,m \geq 1$. Further, since $\psi$ is radial, the projection $\pi_{e\sharp}(\mu \ast \psi_{m})$ has the form $(\pi_{e\sharp}\mu) \ast \phi_{m}$, where $\phi_{m}$ is a bump in $\R$ at scale $2^{-m}$, independent of $e$. Finally, the left hand side of \eqref{form12} is controlled by an absolute constant times $\|(\pi_{e\sharp}\mu) \ast \phi_{m}\|_{2}^{2}$. The inequality now follows by combining all the observations.

Let
\begin{displaymath} C_{e} := 2^{m} \sum_{Q \in \calD_{m}} [\pi_{e\sharp}\mu(Q)]^{2}. \end{displaymath}
Then, for $s < 1$ fixed,
\begin{displaymath} \pi_{e\sharp}\mu\left( \bigcup \left\{Q \in \calD_{m} : \pi_{e\sharp}\mu(Q) \geq 2^{-ms} \right\} \right) \leq C_{e}2^{(s - 1)m}, \end{displaymath}
and so
\begin{equation}\label{form13} \int_{S^{1}} \pi_{e\sharp}\mu\left( \bigcup \left\{Q \in \calD_{m} : \frac{\log \pi_{e\sharp}\mu(Q)}{\log 2^{-m}} \leq s \right\} \right) \, d\sigma(e) \lesssim Am \cdot 2^{(s - 1)m}. \end{equation}
Inspired by \eqref{form13}, let
\begin{displaymath} \calD_{m}^{e-\textup{bad}} := \left\{Q \in \calD_{m} : \frac{\log \pi_{e\sharp}\mu(Q)}{\log 2^{-m}} \leq s \right\}, \end{displaymath}
and denote by $\beta_{e}$ the $\pi_{e\sharp}\mu$-measure of the $e$-bad intervals. Then,
\begin{align*} \int_{S^{1}} H_{m}(\pi_{e\sharp}\mu) \, d\sigma(e) & \geq \int_{S^{1}} \sum_{Q \in \calD_{m} \setminus \calD_{m}^{e-bad}} \pi_{e\sharp}\mu(Q) \left(\frac{\log \pi_{e\sharp}\mu(Q)}{\log 2^{-m}} \right)\\
& \geq \int_{S^{1}} s(1 - \beta_{e}) \, d\sigma(e) \geq s - ACm \cdot 2^{(s - 1)m}, \end{align*} 
as claimed. \end{proof}

\begin{cor}\label{entropyCor} Let $\mu$ be as in Proposition \ref{entropyMarstrand}, and let $S_{2^{m}}$ be a collection of vectors with $|S_{2^{m}}| \sim 2^{m}$ such that every vector $e \in S^{1}$ is at distance $\lesssim 2^{-m}$ from one of the vectors in $S_{2^{m}}$. Then,
\begin{displaymath} \sum_{e \in S_{2^{m}}} p_{e} \cdot H_{m}(\pi_{e\sharp}\mu) \geq s - AC(m \cdot 2^{(s - 1)m} + 1/m), \quad 0 < s < 1, \end{displaymath}
where $p_{e} \sim 2^{-m}$ depends only on $S_{2^{m}}$, and the $C \geq 1$ only depends on the constants behind the $\sim$ and $\lesssim$ notation in the hypothesis. 
\end{cor} 

\begin{proof} For each $e \in S_{2^{m}}$, let
\begin{displaymath} J_{e} := \{\xi \in S^{1} : |\xi - e| \leq C2^{-m}\}. \end{displaymath}
If $C$ is large enough, $S^{1}$ is contained in the union of the arcs $J_{e}$. Let $p_{e} := \sigma(J_{e}) \sim 2^{-m}$, and note that
\begin{displaymath} |H_{m}(\pi_{e_{1}\sharp}\mu) - H_{m}(\pi_{e_{2}\sharp}\mu)| \lesssim \frac{1}{m} \end{displaymath}
by Lemma \ref{factsOfLife}(ii), whenever $e_{1},e_{2} \in J_{e}$ for a fixed $e \in S_{2^{m}}$. Then,
\begin{align*} \sum_{e \in S_{2^{m}}} p_{e} \cdot H_{m}(\pi_{e\sharp}\mu) & = \sum_{e \in S_{2^{m}}} \int_{J_{e}} H_{m}(\pi_{e\sharp}\mu) \, d\sigma(\xi)\\
& \geq \sum_{e \in S_{2^{m}}} \int_{J_{e}} (H_{m}(\pi_{\xi\sharp}\mu) - C/m) \, d\sigma(\xi)\\
& \geq \int_{S^{1}} H_{m}(\pi_{\xi\sharp}\mu) \, d\sigma(\xi) - \sum_{e \in S_{2^{m}}} \frac{Cp_{e}}{m}\\
& \geq s - AC(m \cdot 2^{(s - 1)m} + 1/m), \end{align*} 
where the constant $C$ possibly changed between the last two lines. The proof is complete. \end{proof}

To sum up the progress so far, Lemma \ref{multiScale} shows that, for a rather arbitrary function $f$, we can lower bound the entropy $H_{n}(f_{\sharp}\mu)$ by a linear combination of "partial entropies" of the form $H(f_{\sharp}\mu_{Q},\calD_{(k + 1)m} | \calD_{km})$. Further, if $f$ were the orthogonal projection $\pi_{e}$, a combination of equation \eqref{form17} and Corollary \ref{entropyCor} implies that the terms $H(\pi_{e\sharp}\mu_{Q},\calD_{(k + 1)m} | \calD_{km})$ are rather large on average. Next, we record a fairly standard "error estimate", saying that if $f$ and $\pi_{e}$ are close, then $H(\pi_{e\sharp}\mu_{Q},\calD_{(k + 1)m} | \calD_{km})$ and $H(f_{\sharp}\mu_{Q},\calD_{(k + 1)m} | \calD_{km})$ do not differ much, either:

\begin{lemma}\label{errorEstimate} Assume that $\mu \in \calP([0,1])^{2}$ is supported on a square $Q \in \calD_{km}$, $km \in \N$. Assume that $f \colon \R^{2} \to \R$ is continuously differentiable, and $|\nabla f(x) - e| \leq 2^{-m}$ for all $x \in Q$, and for some fixed vector $e \in S^{1}$. Then
\begin{displaymath} |H(\pi_{e\sharp}\mu,\calD_{(k + 1)m} | \calD_{km}) - H(f_{\sharp}\mu,\calD_{(k + 1)m} | \calD_{km})| \leq C \end{displaymath}
for some absolute constant $C \geq 1$. \end{lemma}
\begin{proof} First note that
\begin{displaymath} H(\pi_{e\sharp}\mu,\calD_{(k + 1)m} | \calD_{km}) = H(\pi_{e\sharp}\mu,\calD_{(k + 1)m}) - H(\pi_{e\sharp}\mu, \calD_{km}) \end{displaymath}
by Proposition \ref{CEF}, and the second term is bounded by a constant, since the support of $\pi_{e\sharp}\mu_{Q}$ is contained in a constant number of cubes (rather: intervals) in $\calD_{km}$. Since $f$ is $2$-Lipschitz on $Q$, similar considerations apply to $H(f_{\sharp}\mu,\calD_{(k + 1)m} | \calD_{km})$. Consequently, it suffices to show that
\begin{displaymath} |H(\pi_{e\sharp}\mu,\calD_{(k + 1)m}) - H(f_{\sharp}\mu_{Q},\calD_{(k + 1)m})| \leq C. \end{displaymath}
To this end, by Lemma \ref{factsOfLife} (iii) and the assumption on the support of $\mu$, one needs to prove that any set of the form $f^{-1}(I) \cap Q$ with $I \in \calD_{(k + 1)m}$ can be covered by a constant number of sets $\pi_{e}^{-1}(I') \cap Q$, $I' \in \calD_{(k + 1)m}$ and vice versa. Fix $I \in \calD_{(k + 1)m}$ and consider $x,y \in f^{-1}(I) \cap Q$. Then, for some $\xi$ on the line segment connecting $x$ and $y$, one has
\begin{displaymath} |\pi_{e}(x) - \pi_{e}(y)| = |(x - y) \cdot (e - \nabla f(\xi))| + |f(x) - f(y)| \lesssim 2^{-(k + 1)m}. \end{displaymath}
This proves that $f^{-1}(I)$ is contained in a bounded number of sets $\pi_{e}^{-1}(I')$, $I' \in \calD_{(k + 1)m}$, and the converse inclusion is verified similarly. \end{proof}

\section{Proof of the main theorem}\label{mainProof}
We are prepared to prove Theorem \ref{main}. We first do so for upper box dimension instead of packing dimension. In other words, we prove the following theorem:
\begin{thm}\label{mainBox} Assume that $\emptyset \neq K \subset \R^{2}$ is a bounded $\calH^{s}$-measurable $s$-AD-regular set with $s \geq 1$. Then $\Bd D(K) = 1$.
\end{thm}

The proof of Theorem \ref{mainBox} contains all the main ideas needed for Theorem \ref{main}. The passage from box dimension to packing dimension only requires a fairly standard argument using Baire's theorem. The details can be found in Section \ref{packingDimension}. 

\begin{proof}[Proof of Theorem \ref{mainBox}] Let $s \geq 1$, and let $K \subset B(0,1)$ be an $s$-AD-regular set with $\mathcal{H}^{s}(K) > 0$. If $s = 1$ and $K$ contains a non-trivial rectifiable part, then a result of Besicovitch and Miller \cite{BM} from 1948 tells us immediately that $D(K)$ has positive length: in particular $\Pd D(K) = \Bd D(K) = 1$. So, one may assume that either $s > 1$, or $K$ is purely $1$-unrectifiable. In both cases, the following holds for $\mathcal{H}^{s}$ almost all $x \in K$: the set of directions
\begin{displaymath} R_{x} := \left\{\frac{y - x}{|y - x|} : y \in K \setminus \{x\} \right\} \end{displaymath}
is dense in $S^{1}$. For $s > 1$, this is a consequence of Marstrand's classical slicing result, see \cite{Ma} or the exposition in Mattila's book \cite[Chapter 10]{Mat}. For $s = 1$ the claim follows from unrectifiability, see (the proof of) \cite[Lemma 15.13]{Mat}. For convenience, assume that $0 \in K$, and $R_{0}$ is dense in $S^{1}$. 

Let $m \in \N$ be a large integer to be specified later, and choose a $2^{-m}$-net of vectors $S_{2^{m}} \subset R_{0}$. Also, for each $e \in S_{2^{m}}$, fix a point $x_{e} \in K \setminus \{0\}$ such that $x_{e}/|x_{e}| = e$. With each such $x_{e}$, associate the mapping $f_{e}(y) := |y - x_{e}|^{2}/(2|x_{e}|)$. The point of this definition is that
\begin{equation}\label{form20} \nabla f_{e}(y) = \frac{y - x_{e}}{|x_{e}|}, \end{equation}
so if $y$ ranges in a small square $Q_{0}$ with $\dist(0,Q_{0}) \ll |x_{e}|$ and $\ell(Q_{0}) \ll |x_{e}|$, then $\nabla f_{e}(y)$ deviates only a bit from the vector $-x_{e}/|x_{e}| \in -S_{2^{m}}$. 

Now fix a small dyadic square $Q_{0}$ such that $\overline{Q_{0}}$ contains the origin and such that $\mathcal{H}^{s}(K \cap Q_{0}) \geq \ell(Q_{0})^{s}/(10 A)$; this is possible since $0 \in K$. Fixing also $0 < t < 1$, the claim is that that if $m$ is large enough (depending on $t$ and the AD-regularity constant $A$ of $K$), and $Q_{0}$ is small enough (depending on the lengths of the vectors $x_{e}$ selected after $m$ was chosen), we can prove that
\begin{equation}\label{form18} \sum_{e \in S_{2^{m}}} p_{e} \cdot N(f_{e}(K \cap Q_{0}),\delta) \geq \delta^{-t} \end{equation} 
for all small enough $\delta > 0$ (depending on $\ell(Q_{0})$). The constants $p_{e} \sim 2^{-m}$ were defined in Corollary \ref{entropyCor}. Inequality \eqref{form18} clearly implies Theorem \ref{mainBox}. 

To prove \eqref{form18}, let 
\begin{displaymath} \mu := (\calH^{s}|_{K})_{Q_{0}} = \frac{1}{\calH^{s}(K \cap Q_{0})}\calH^{s}|_{K \cap Q_{0}}. \end{displaymath}
Recalling that $\calH^{s}(K \cap Q_{0}) \geq \ell(Q_{0})^{s}/(10A)$, the measure $\mu$ is a well-defined probability measure supported on $Q_{0}$. Apply Lemma \ref{multiScale} to the mappings $f_{e}$ to deduce that
\begin{displaymath}  \sum_{e \in S_{2^{m}}} p_{e} \cdot H_{n}(f_{e\sharp}\mu) \geq \frac{1}{n} \sum_{k = 0}^{\floor*{n/m}} \sum_{Q \in \calD_{km}} \mu(Q) \left[ \sum_{e \in S_{2^{m}}} p_{e} \cdot H(f_{e\sharp}\mu_{Q},\calD_{(k + 1)m} | \calD_{km}) \right]. \end{displaymath}
To get our hands on $\pi_{e}$ instead of $f_{e}$, we apply Lemma \ref{errorEstimate}, which says that
\begin{equation}\label{form19}  H(f_{e\sharp}\mu_{Q},\calD_{(k + 1)m} | \calD_{km})  \geq  H(\pi_{e'\sharp}\mu_{Q},\calD_{(k + 1)m} | \calD_{km}) - C, \end{equation}
as soon as $|\nabla f_{e}(y) - e'| \leq 2^{-m}$ for $y \in Q_{0}$. With $e' = -e$, equation \eqref{form20} guarantees that this can be arranged by choosing $\ell(Q_{0})$ small enough depending only on $|x_{e}|$ and $m$ (we will have no needs for the precise bounds on $\ell(Q_{0})$). After applying \eqref{form19} and \eqref{form17}, we end up with
\begin{align} \sum_{e \in S_{2^{m}}} p_{e} \cdot H_{n}(f_{e\sharp}\mu) & \geq \frac{1}{n} \sum_{k = 0}^{\floor*{n/m}} \sum_{Q \in \calD_{km}} \mu(Q) \left[ \sum_{e \in -S_{2^{m}}} p_{e} \cdot (m \cdot H_{m}(\pi_{e\sharp}\mu^{Q}) - 3 - C) \right] \notag\\
&\label{form21} \geq \left[\frac{m}{n} \sum_{k = 0}^{\floor*{n/m}} \sum_{Q \in \calD_{km}} \mu(Q) \sum_{e \in -S_{2^{m}}} p_{e} \cdot H_{m}(\pi_{e\sharp}\mu^{Q}) \right] - \frac{3C}{m}, \end{align} 
assuming that $C \geq 3$ in \eqref{form19} and recalling that $p_{e} \sim 1/|S_{2^{m}}|$. 

Next, a straightforward calculation shows that that for $Q \subset Q_{0}$ with $\mu(Q) > 0$, the measures $\mu^{Q}$ satisfy the growth condition $\mu^{Q}(B(x,r)) \leq C_{Q}r^{s}$, $0 < r \leq 1$, for some constant $C_{Q} \lesssim A\ell(Q)^{s}/\calH^{s}(K \cap Q)$.\footnote{One cannot quite say that the measures $\mu^{Q}$ are $(s,C_{Q})$-AD-regular, because the lower bound on the measure of balls may fail close to the boundary of $[0,1)^{2}$.} Here $A$ is the AD-regularity constant of $K$. In particular, since $s \geq 1$, they also satisfy the linear growth condition
\begin{displaymath} \mu^{Q}(B(x,r)) \leq \left(\frac{AC\ell(Q)^{s}}{\calH^{s}(K \cap Q)}\right) r, \qquad 0 < r \leq 1. \end{displaymath} 
Let $k_{0}$ be the smallest integer such that $2^{-k_{0}m} \leq \ell(Q_{0})$ (this implies that if $k \geq k_{0}$ and $Q \in \calD_{km}$ satisfies $\mu(Q) > 0$, then $Q \subset Q_{0}$, and hence the previous discussion is valid). Applying Corollary \ref{entropyCor} with $-S_{2^{m}}$ in place of $S_{2^{m}}$, and some $t' \in (t,1)$, the quantity in brackets on line \eqref{form21} is bounded from below as follows:
\begin{align*} \frac{m}{n} & \sum_{k = k_{0}}^{\floor*{n/m}} \mathop{\sum_{Q \in \calD_{km}}}_{\mu(Q) > 0} \mu(Q) \left(t' - \left(\frac{AC\ell(Q)^{s}}{\calH^{s}(K \cap Q)} \right)(m \cdot 2^{(t' - 1)m} + \frac{1}{m})\right)\\
& = t' \cdot \frac{m}{n} \cdot (\floor*{n/m} - k_{0}) - \frac{ACm}{\calH^{s}(K \cap Q_{0})n} \sum_{k = k_{0}}^{\floor*{n/m}} \mathop{\sum_{Q \in \calD_{km}}}_{\mu(Q) > 0} \ell(Q)^{s}\left(m \cdot 2^{(t' - 1)m} + \frac{1}{m} \right). \end{align*} 

To proceed further, observe that, for any fixed generation of squares $Q$ with $\ell(Q) = r \leq \ell(Q_{0})$, there are at most $A^{3}C\calH^{s}(K \cap Q_{0})/r^{s}$ squares $Q \subset Q_{0}$ such that $\mu(Q) > 0$. Indeed, by the $(s,A)$-AD regularity of $K$, each square $Q$ with $\mu(Q) > 0$ is adjacent to a "good" square $Q'$ with $\ell(Q') = r$ and $\calH^{s}(K \cap Q') \geq r^{s}/(10A)$. It is possible that such a $Q'$ lies outside $Q_{0}$, but it is certainly contained in $2Q_{0}$; thus, recalling that $\calH^{s}(K \cap Q_{0}) \geq \ell(Q_{0})^{s}/(10A)$, the number of "good" squares is bounded by
\begin{displaymath} \frac{10 A\calH^{s}(K \cap 2Q_{0})}{r^{s}} \leq \frac{100 A^{2}\ell(Q_{0})^{s}}{r^{s}} \leq \frac{1000 A^{3}\calH^{s}(K \cap Q_{0})}{r^{s}}. \end{displaymath}
Since each of the "good" square $Q'$ is adjacent to at most eight squares $Q$ with $\mu(Q) > 0$, the number of these $Q$ is bounded by $8000A^{3}\calH^{s}(K \cap Q_{0})/r^{s}$.

Combining everything so far, one has the estimate
\begin{displaymath} \sum_{e \in S_{2^{m}}} p_{e} \cdot H_{n}(f_{e\sharp}\mu) \geq t' \cdot \frac{m}{n} \cdot (\floor*{n/m} - k_{0}) - \frac{A^{4}Cm}{n} \cdot \floor*{n/m} \cdot (m \cdot 2^{(t' - 1)m} + 1/m) - \frac{3C}{m}, \end{displaymath}
valid for any $t \leq t' < 1$ and $n$ so large that $n/m > k_{0}$. Specialising to $t' := (1 + t)/2$, say, and choosing $m \geq m(A,t)$, where $m(A,t) \in \N$ depends only on $A$ and $t$, one obtains
\begin{displaymath} \sum_{e \in S_{2^{m}}} p_{e} \cdot H_{n}(f_{e\sharp}\mu) \geq t  \end{displaymath}
for large enough $n$ (so large that $(m/n) \cdot (\floor*{n/m} - k_{0})$ is close enough to $1$). Via the following lemma, this implies \eqref{form18}. The proof of Theorem \ref{mainBox} is complete.
\end{proof}

\begin{lemma}\label{entropyAndCovering} Let $\nu \in \calP(\R^{d})$, and assume that $H_{n}(\nu) \geq s$. Then
\begin{displaymath} |\{Q \in \calD_{n} : \nu(Q) > 0\}| > 2^{nt} \end{displaymath}
for any $t < s - 1/(n \log 2)$. In particular, $N(\spt \nu,2^{-n}) \gtrsim 2^{nt}$ for such $t$.
\end{lemma}

\begin{remark} Note that the converse of the lemma is false: a large covering number certainly does not guarantee large entropy. 
\end{remark}

\begin{proof}[Proof of Lemma \ref{entropyAndCovering}] Assume that $|\{Q \in \calD_{n} : \nu(Q) > 0\}| \leq 2^{nt}$ for some $t$, and let $\calD_{n}^{\lambda-\textup{bad}}$, $\lambda \geq 0$, be the cubes $Q \in \calD_{n}$ such that $\nu(Q) \leq 2^{-\lambda n}$. Then
\begin{displaymath} \sum_{Q \in \calD_{n}^{\lambda-\textup{bad}}} \nu(Q) \leq 2^{(t - \lambda)n}, \qquad \lambda \geq t, \end{displaymath}
so that
\begin{align*} s \leq H_{n}(\nu) & = \int_{0}^{\infty} \nu\left( \bigcup \left\{Q : \frac{\log \nu(Q)}{\log 2^{-n}} \geq \lambda \right\} \right) \, d\lambda\\
& \leq t + \int_{t}^{\infty} \nu \left( \bigcup \left\{ Q : \nu(Q) \leq 2^{-\lambda n} \right\} \right) \, d\lambda\\
& \leq t + \int_{t}^{\infty} 2^{(t - \lambda)n} \, d\lambda = t + \frac{1}{n \log 2}. \end{align*}
This proves the lemma. \end{proof}

\subsection{Packing dimension}\label{packingDimension} The purpose of this section is to complete the proof of Theorem \ref{main}. This builds heavily on the proof of Theorem \ref{mainBox} from the previous section. The inequality \eqref{form18} certainly implies that 
\begin{equation}\label{form22} \max_{e \in S_{2^{m}}} \Bd f_{e}(K \cap Q_{0}) \geq t. \end{equation}
The choice of the square $Q_{0}$ was somewhat arbitrary, and the same proof gives the following just as well: if $Q \subset Q_{0}$ is any square with $\calH^{s}(K \cap Q) \geq \ell(Q)^{s}/(10A)$, then \eqref{form22} holds with $Q$ in place of $Q_{0}$. Now, we claim
\begin{displaymath} \max_{e \in S_{2^{m}}} \Pd f_{e}(K \cap \overline{Q_{0}}) \geq t, \end{displaymath}
assuming that $K$ is compact and purely unrectifiable. This implies that $\Pd D(K) = 1$, since $t < 1$ is arbitrary. The pure unrectifiability will only be used here to infer that boundaries of squares are $\calH^{s}|_{K}$-null. Since $K$ is compact and the mappings $f_{e}$ are continuous, it suffices to prove that
\begin{equation}\label{form23} \max_{e \in S_{2^{m}}} \Pd \overline{f_{e}(K \cap \interior Q_{0})} \geq t. \end{equation}
Note that $K \cap \interior Q_{0} \neq \emptyset$, since otherwise $\calH^{s}(K \cap \partial Q_{0}) > 0$. Recall that
\begin{displaymath} \dim_{\textup{p}}(B) = \inf \left\{ \sup \Bd F_{i} : B \subset \bigcup_{i} F_{i}, \: F_{i} \text{ closed}\right\}. \end{displaymath} 
So, if \eqref{form23} fails, one may find $t' < t$ such that each set $\overline{f_{e}(K \cap \interior Q_{0})}$, $e \in S_{2^{m}}$, can be covered by closed sets $F_{i}^{e}$, $i \in \N$, satisfying the uniform bound $\Bd F_{i}^{e} \leq t'$. 

Write $S_{2^{m}} := \{e_{1},\ldots,e_{N}\}$. We first study $\overline{f_{e_{1}}(K \cap \interior Q_{0})}$. Since $\overline{f_{e_{1}}(K \cap \interior Q_{0})}$ is compact and not empty, Baire's theorem says that one of the sets $F_{i}^{e_{1}}$, say $F_{i_{1}}^{e_{1}}$, has non-empty interior in the relative topology of $\overline{f_{e_{1}}(K \cap \interior Q_{0})}$. In other words, there exists an open set $U^{e_{1}} \subset \R$ such that
\begin{displaymath} \emptyset \neq U^{e_{1}} \cap \overline{f_{e_{1}}(K \cap \interior Q_{0})} \subset F_{i_{1}}^{e_{1}}. \end{displaymath}
Now, find a dyadic square $Q_{1} \subset \interior Q_{0}$ with $\calH^{s}(K \cap Q_{1}) \geq \ell(Q_{1})^{s}/(10A)$ such that 
\begin{equation}\label{form24} \overline{f_{e_{1}}(K \cap Q_{1})} \subset U^{e_{1}} \cap \overline{f_{e_{1}}(K \cap \interior Q_{0})} \subset F_{i_{1}}^{e_{1}}. \end{equation}
This is easy: since $U^{e_{1}}$ is open, one can find $x \in K \cap \interior Q_{0}$ such that $f(x) \in U^{e_{1}} \cap f_{e_{1}}(K \cap \interior Q_{0})$. Then, the AD-regularity of $K$ shows that the desired square $Q_{1} \subset \interior Q_{0}$ can be found inside $B(x,r) \subset \interior Q_{0}$ for some small enough radius $r > 0$. 

The next step is to apply \eqref{form22} (rather the discussion just below it) to the square $Q_{1}$: there exists some index $j \in \{1,\ldots,N\}$ such that 
\begin{displaymath} \Bd f_{e_{j}}(K \cap Q_{1}) \geq t. \end{displaymath}
Certainly $j \neq 1$ by \eqref{form24} and the uniform bound $\Bd F_{i}^{e} \leq t' < t$. 

Assume for instance that $j = 2$. The set $K \cap \interior Q_{1}$ is non-empty, and $\overline{f_{e_{2}}(K \cap \interior Q_{1})}$ is covered by the sets $F_{i}^{e_{2}}$. By Baire's theorem again, there is an index $i_{2}$ and an open set $U^{e_{2}}$ such that
\begin{displaymath} \emptyset \neq U^{e_{2}} \cap \overline{f_{e_{2}}(K \cap \interior Q_{1})} \subset F_{i_{2}}^{e_{2}}. \end{displaymath}
And, as before, one can find a square $Q_{2} \subset \interior Q_{1}$ such that $\calH^{s}(K \cap Q_{2}) \geq \ell(Q_{2})^{s}/(10A)$ and
\begin{equation}\label{form25} \overline{f_{e_{2}}(K \cap Q_{2})} \subset U^{e_{2}} \cap \overline{f_{e_{2}}(K \cap \interior Q_{1})} \subset F_{i_{2}}^{e_{2}}. \end{equation}
Once again, there exists an index $j \in \{1,\ldots,N\}$ such that 
\begin{displaymath} \Bd f_{e_{j}}(K \cap Q_{2}) \geq t. \end{displaymath}
But now $j \notin \{1,2\}$ by \eqref{form24} and \eqref{form25} combined.

Repeating the same argument for some $N$ steps eventually produces a square $Q_{N} \subset Q_{0}$ with $\calH^{s}(K \cap Q_{N}) \geq \ell(Q_{N})^{s}/(10A)$, and $f_{e_{j}}(K \cap Q_{N}) \subset F^{e_{j}}_{i_{j}}$ for all $1 \leq j \leq N$, and for some $i_{j} \in \N$. This contradicts the fact that
\begin{displaymath} \max_{j} \Bd f_{e_{j}}(K \cap Q_{N}) \geq t, \end{displaymath}
and the proof of Theorem \ref{main} is complete.

\section{Further results}\label{corollaries}

In this section, I discuss the proofs of Theorem \ref{dotProducts} and Corollary \ref{sumProduct}. One first proves a box dimension variant of Theorem \ref{dotProducts}:
\begin{thm} Assume that $\emptyset \neq K \subset \R^{2}$ is a bounded $\calH^{s}$-measurable $s$-AD-regular set with $s \geq 1$. Then, there exists a point $x_{0} \in K$ such that
\begin{displaymath} \Bd K \cdot (K - x_{0}) = 1. \end{displaymath}
\end{thm}

\begin{proof} The proof is extremely similar to that of Theorem \ref{mainBox}, even a bit easier. Once again, the result is elementary if $s = 1$ and $K$ contains a non-trivial rectifiable part: in this case $\calH^{1}(K \cdot (K - x_{0})) > 0$ for any $x_{0} \in K$. Indeed, either $K$ lies on a line spanned by some vector $e \in S^{1}$, and then $K \cdot (K - x_{0}) \supset x \cdot (K - x_{0})$ contains an affine copy of $K$ for any $x \in K \setminus \{0\}$. Or else one can find two linearly independent vectors $x_{1},x_{2} \in K$, in which case one of the orthogonal projections $\pi_{x_{1}/|x_{1}|}(K - x_{0})$ or $\pi_{x_{2}/|x_{2}|}(K - x_{0})$ has positive length, since otherwise $K - x_{0}$ (hence $K$ itself) would be purely $1$-unrectifiable. Then, it suffices to note that $K \cdot (K - x_{0})$ contains an affine copy of both of these projections.

So, one may assume that either $s =1$ and $K$ is purely $1$-unrectifiable, or $s > 1$. In both cases, there exists $x_{0} \in K$ such that 
\begin{displaymath} R_{x_{0}} := \left\{\frac{y - x_{0}}{|y - x_{0}|} : y \in K \setminus \{x_{0}\} \right\} \end{displaymath}
is dense in $S^{1}$.\footnote{In the proof for distance sets, one could assume that $x_{0} = 0$ at this stage. Here one cannot, since there is no obvious reason why the problem would be translation invariant. The erroneous "without loss of generality $x_{0} = 0$" statement made its way to an earlier version of the manuscript, and I am grateful to Pablo Shmerkin for pointing out the issue. If $\overline{R_{0}}$ happens to have non-empty interior, then the proof could be modified to show that $\Pd K \cdot K = 1$.} Then, one fixes $t \in (0,1)$ and $m \in \N$ and picks a set of vectors $S_{2^{m}} \subset R_{x_{0}}$ as in Corollary \ref{entropyCor}. For each $e \in S_{2^{m}}$, one locates a point $x_{e} \in K \setminus \{x_{0}\}$ such that $(x_{e} - x_{0})/|x_{e} - x_{0}| = e$, and then one considers the the family of orthogonal projections $\pi_{e}$, $e \in S_{2^{m}}$. Since
\begin{displaymath} \pi_{e}(K) = K \cdot e = \frac{K \cdot (x_{e} - x_{0})}{|x_{e} - x_{0}|} \subset \frac{K \cdot (K - x_{0})}{|x_{e} - x_{0}|}, \end{displaymath}
it suffices to prove that
\begin{equation}\label{form26} \sum_{e \in S_{2^{m}}} p_{e} \cdot N(\pi_{e}(K),\delta) \geq \delta^{-t} \end{equation}
for all small enough $\delta > 0$, if $m = m(A,t) \in \N$ was chosen large enough. This time one does not even need to restrict $K$ to a small square about the origin, and the rest of the proof runs exactly in the same was as that of Theorem \ref{mainBox}.
\end{proof}

The argument for the packing dimension version of Theorem \ref{dotProducts} is the same as given in Section \ref{packingDimension}, and there is no point in repeating the details.
\begin{proof}[Proof of Corollary \ref{sumProduct}] Assume that $1/2 \leq s \leq 1$, and $A \subset \R$ is a compact AD-regular set with $\Hd A = s$. Then $A \times A \subset \R^{2}$ is a compact $2s$-AD-regular set, since the product measure $\calH^{s}|_{A} \times \calH^{s}|_{A}$ is equivalent to $\calH^{2s}|_{A \times A}$ with uniform constants. It now follows from Theorem \ref{dotProducts} that there exists a point $(a_{1},a_{2}) \in A \times A$ such that
\begin{displaymath} 1 = \Pd (A \times A) \cdot (A \times A - (a_{1},a_{2})) = \Pd [A(A - a_{1}) + A(A - a_{2})], \end{displaymath}
as claimed. \end{proof}

\end{document}